\documentclass[a4paper]{amsproc}     

\usepackage{graphicx}
\usepackage[labelfont=bf, labelsep=space]{caption}
\usepackage{amsmath, amssymb, amsfonts, textcomp, mathrsfs, dirtytalk, float, ulem}
\newtheorem{thm}{Theorem}
\newtheorem{lem}{Lemma}
\newtheorem{clm}{Claim}
\newtheorem{rem}{Remark}
\newtheorem{cor}{Corollary}
\newtheorem{ques}{Question}
\newcommand{\oga}{\overline{\gamma}}
\newcommand{\tga}{T_\gamma(\alpha)}

\begin{document}

\title{Distance $4$ Curves on Closed Surfaces of Arbitrary Genus
}

\author{Kuwari Mahanta \and Sreekrishna Palaparthi*} 
\maketitle

\begin{abstract}
Let $S_g$ denote a closed, orientable surface of genus $g \geq 2$ and $\mathcal{C}(S_g)$ be the associated curve complex. The mapping class group of $S_g$, $Mod(S_g)$ acts on $\mathcal{C}(S_g)$ by isometries. Since Dehn twists about certain curves generate $Mod(S_g)$, one can ask how Dehn twists move specific vertices in $\mathcal{C}(S_g)$ away from themselves. We show that if two curves represent vertices at a distance $3$ in $\mathcal{C}(S_g)$ then the Dehn twist of one curve about another yields two vertices at distance $4$. This produces many tractable examples of distance $4$ vertices in $\mathcal{C}(S_g)$. We also show that the minimum intersection number of any two curves at a distance $4$ on $S_g$ is at most $(2g-1)^2$.
\end{abstract}

\section{Introduction}
\label{section:introduction}
Let $S_g$ denote a closed, orientable surface of genus $g \geq 2$. Throughout this article, a curve on $S_g$ will mean an essential simple closed curve on it. For two curves, $\alpha, \beta$ in $S_g$, $i(\alpha,\beta)$ denotes their geometric intersection number and $T_\alpha(\beta)$ denotes the Dehn twist of the curve $\beta$ about the curve $\alpha$. In \cite{H1}, Harvey associated with $S_g$ a simplical complex called the complex of curves which is denoted by $\mathcal{C}(S_g)$ and defined as follows. The $0$-skeleton, $\mathcal{C}^0(S_g)$ of this complex is in one-to-one correspondence with isotopy classes of essential simple closed curves on $S_g$.  Two vertices span an edge in $\mathcal{C}(S_g)$ if and only if these vertices have mutually disjoint representatives. $\mathcal{C}^0(S_g)$ can be equipped with a metric, $d$ by defining the distance between any two vertices to be the minimum number of edges in any edge path between them in $\mathcal{C}(S_g)$. By the distance between two curves on $S_g$, we mean the distance between the corresponding vertices in $\mathcal{C}(S_g)$.  The minimal intersection number between any two curves on $S_g$ which are at a distance $n$ is denoted by $i_{min}(g,n)$.

\begin{figure}
\centering
\includegraphics[scale=0.2]{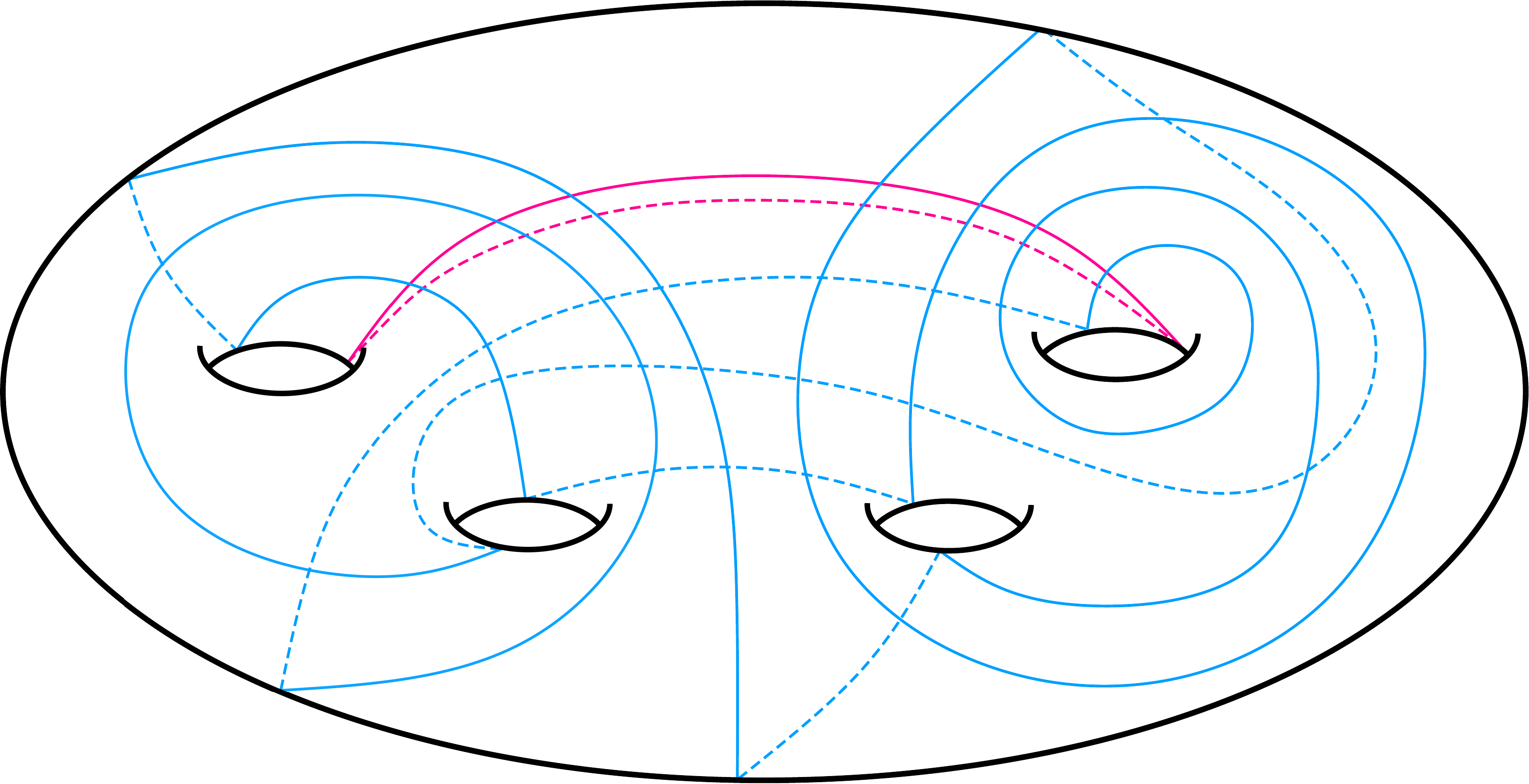}
\caption{Distance $3$ vertices in $\mathcal{C}(S_4)$}
\label{fig:dist3_eg}
\end{figure}

\begin{figure}
\centering
\includegraphics[scale=0.2]{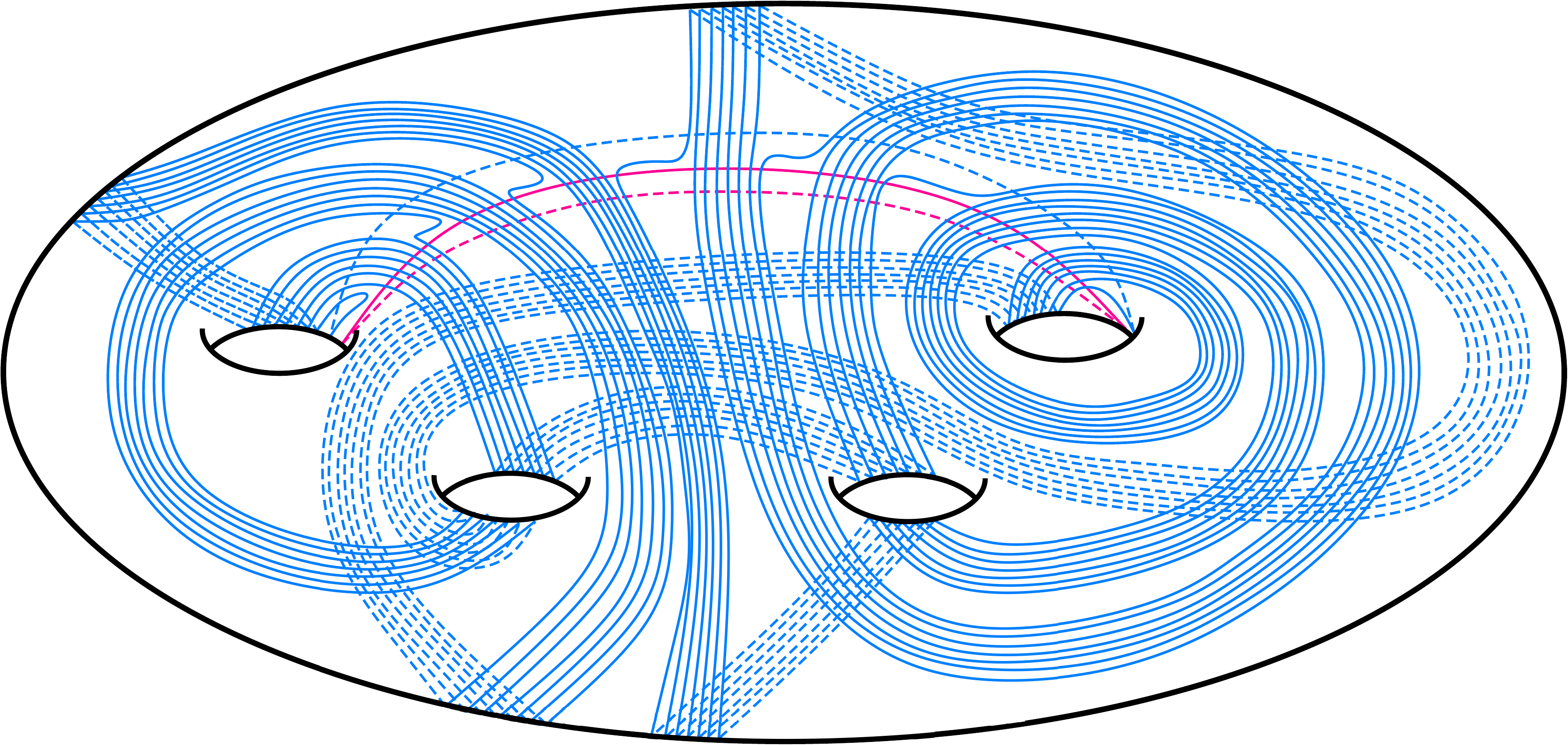}
\caption{Distance $4$ vertices in $\mathcal{C}(S_4)$}
\label{fig:dist4_eg}
\end{figure}

Masur and Minsky, in their seminal paper \cite{MM2}, proved that $\mathcal{C}(S)$ is a $\delta$-hyperbolic space with the metric $d$. Later, it was shown that the $\delta$ can be chosen to be independent of the surface $S_g$, see \cite{A1}, \cite{B1}, \cite{CRS}, \cite{HPW}, \cite{PS}. The coarse geometry of the curve complex has various applications to $3$-manifolds, Teichmuller theory and mapping class groups. One can see \cite{M1} for many such applications.

\cite{S1}, \cite{W1}, \cite{W2} and \cite{BMM} gave algorithms to compute distances between two vertices in $\mathcal{C}(S_g)$. The algorithm in \cite{BMM}, although for closed surfaces is by far the most effective in calculating distances in $\mathcal{C}(S_g)$. Given vertices, $\nu$, $\mu$ in $\mathcal{C}(S_g)$, the algorithm works by giving an initially efficient geodesic, $\nu_0 = \nu$, $\nu_1$, \dots, $\nu_n=\mu$ where $\nu_1$ is chosen from a list of $n^{6g-6}$ possible vertices.

In \cite{AH1}, Aougab and Huang gave all $Mod(S_g)$ orbits of minimally intersecting distance $3$ curves in $\mathcal{C}(S_g)$ by giving the curves as permutations which are solutions to a system of equation in the permutation group of order $8g-4$. They further showed that for genus, $g \geq 3$, $i_{min}(g,3) = 2g-1$ and that all pairs of curves which intersect at $2g-1$ points and cuts up $S_{g\geq 3}$ into a disc are at distance $3$.

For vertices at distance $4$ in $\mathcal{C}(S_g)$, we found only limited pictures (like, Figure \ref{fig:dist4_eg}) of such curves in $S_{g\leq3}$ and none for $S_{g>3}$. In \cite{GMMM}, the authors gave a test to determine when two vertices in $\mathcal{C}(S_g)$ are at a distance $\geq 4$ using the efficient geodesic algorithm of \cite{BMM}. Using the MICC software, they showed that $i_{min}(2,4) =12$ by giving all minimally intersecting pairs of curves at distance $4$. They also gave examples of curves which are at a distance $4$ on a surface of genus $3$ and concluded that $i_{min}(3,4) \leq 29$ (refer \cite{GMMM} Theorem 1.8).

Although $i_{min}(g,4)$ is still not known, in \cite{AT1}, Aougab and Taylor proved that $i_{min}(g,4) = O(g^2)$ by answering a more general question by Dan Margalit that $i_{min}(g,n) = O(g^{n-2})$.

In this article, we give a method to construct examples of vertices at distance $4$ in $\mathcal{C}(S_g)$ and improve the known upper bound of $i_{min}(g,4)$ to $(2g-1)^2$. In particular, using any of the minimally intersecting filling pairs of curves on $S_3$, as described in \cite{AH1}, one can get examples of curves at distance $4$ in $\mathcal{C}(S_3)$ which intersect at $25$ points. This implies that $i_{min}(3,4) \leq 25$. Using a minimally intersecting pair of curves (Figure \ref{fig:dist3_eg}) described in \cite{AH1} as a permutation of $28$ symbols, we use our method to construct a pair of distance $4$ curves in $\mathcal{C}(S_4)$ as in Figure \ref{fig:dist4_eg}.

Such workable examples of distance $4$ curves were a result of trying to investigate the effect of certain Dehn twists on distances in $\mathcal{C}(S_g)$ as follows:
\begin{rem}
If $d(\alpha,\gamma) = 1$ or $2$, then $d(\alpha, \tga) = d(\alpha,\gamma)$.
\end{rem}
In general, one can ask the following question:

\begin{ques}\label{question:q1} If $d(\alpha, \gamma) > 2$, then what is the relation between $d(\alpha, \gamma)$ and $d(\gamma, T_\gamma(\alpha))$?\end{ques}

In this article, we prove that if $d(\alpha, \gamma)\geq 3$ then $d(\alpha, \tga) \geq 3$. We then answer the question \ref{question:q1} for $d(\alpha, \gamma) =3$ and show that $d(\alpha, T_\gamma(\alpha)) = d(\alpha, \gamma) +1$. 

\section{Setup}
\label{section:special_position}
For any ordered index in this work, we follow cyclical ordering. For instance, if $i \in \{1, 2, \dots, k\}$, $i=k+1$ will indicate $i=1$.

Two curves, $\nu_1$ and $\nu_2$ on $S_g$ are said to be a filling pair if every component of $S_g \setminus \{\nu_1, \nu_2\}$ is a disk. A component, $D$ of $S_g \setminus \{\nu_1, \nu_2\}$ is said to be an $2n$-gon if its boundary comprises of $n$ arcs of $\nu_1$ and $\nu_2$. Consider a pair, $\alpha, \beta$, of filling curves on $S_g$ with geometric intersection number, $i(\alpha, \beta)=k$. Let their set of intersection points be $\{w_1, \dots, w_k\}$. Define a triangulation, $G$ of $S$ as follows:

Let $\{w_1', \dots, w_k'\}$ be the $0$-skeleton of $G$. There is an edge between $w_i'$ and $w_j'$ in $G$ if and only if there is an arc of $\alpha$ or, $\beta$ between the two intersection points, $w_i, w_j$ in $S$. As $\alpha$ and $\beta$ are filling pairs, for each disk component, $D$ of $S_g \setminus \{\alpha, \beta\}$ attach a disk to the cycle formed by the edges of $G$ corresponding to the arcs of $\alpha, \beta$ that form the boundary of $D$. Using the Euler characteristic of a $2$-dimensional complex, the number of faces, $f$ of $G$ is given by $f = k+2-2g$.

Consider a geodesic, $\nu_0, \dots, \nu_N$ of length $N$ in $\mathcal{C}^0(S)$. An arc, $\omega$ in $S$ is a \textit{reference arc} for the triple $\nu_0$, $\nu_1$, $\nu_N$ if $\omega$ and $\nu_1$ are in minimal position and the interior of $\omega$ is disjoint from $\nu_0 \cup \nu_N$. The oriented geodesic $\nu_0, \dots, \nu_N$ is said to be \textit{initially efficient} if $i(\nu_1, \omega) \leq N-1$ for all choices of reference arc, $\omega$. The authors of \cite{BMM} prove that there exists an initially efficient geodesic between any two vertices of $\mathcal{C}(S)$.

\begin{figure}
\centering
\includegraphics[scale=0.55]{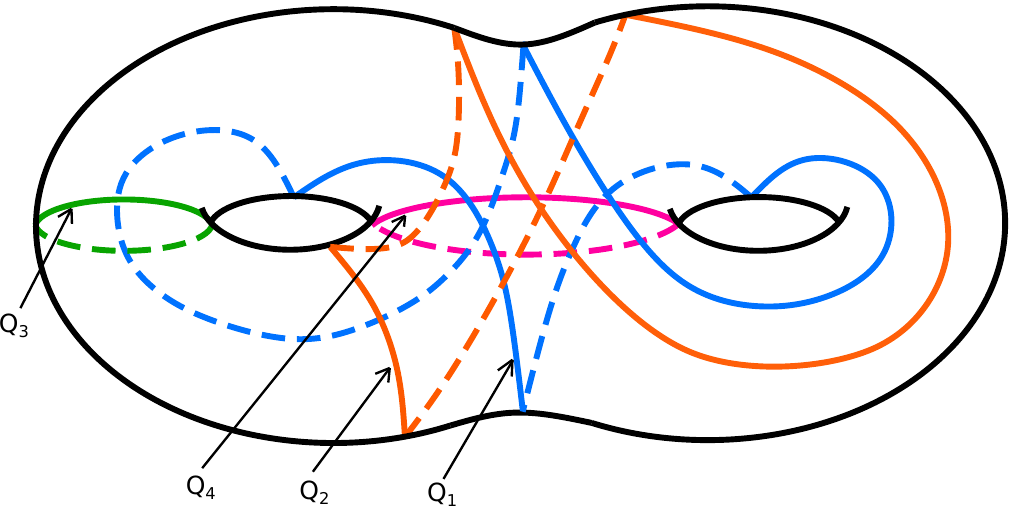}
\caption{$Q_1, Q_2, Q_3, Q_4$ is an example of an initially efficient geodesic in $\mathcal{C}(S_2)$}
\label{fig_gen_2_dist_3_eg}
\end{figure}

The following theorem from \cite{GMMM} gives a criterion for detecting vertices in $\mathcal{C}(S)$ at distance at-least $4$.

\begin{thm}[Theorem 1.3, \cite{GMMM}]\label{thm:dist_geq_4}
For the filling pair, $\kappa$, $\omega$, let $\Gamma \subset \mathcal{C}^0(S)$ be the collection of all vertices such that the following hold :
\begin{enumerate}
\item for $\overline{\gamma} \in \Gamma$, $d(\kappa, \overline{\gamma}) = 1$; and
\item for $\overline{\gamma} \in \Gamma$; for each segment, $b \subset \omega \setminus \kappa$, $i(\overline{\gamma}, b) \leq 1$.
\end{enumerate}
Then $d(\kappa, \omega) \geq 4$ if and only if $d(\overline{\gamma}, \omega) \geq 3$ for all $\overline{\gamma} \in \Gamma$.
\end{thm}

Let $\lambda$ and $\mu$ be two simple closed curves on $S_g$ and let $R_\lambda$ and $R_\mu$ be closed regular neighborhoods of $\lambda$ and $\mu$ respectively. We say that the $4$-tuple $(\lambda, \mu, R_\lambda, R_\mu)$ is \textit{amenable to Dehn twist in special position} if the following hold:
\begin{enumerate}
 \item $\lambda$ and $\mu$ intersect transversely and minimally on $S_g$,
 \item $\lambda$ and $\mu$ fill $S_g$,
 \item the number of components of $R_\lambda \cap R_\mu$ is equal to the number of components of $S \setminus (\lambda \cup \mu)$ and each of these components is a disc.
\end{enumerate}
When $\lambda$ and $\mu$ fill $S_g$, by considering a Euclidean model of $S_g$, it is easy to see that a $4$-tuple $(\lambda, \mu, R_\lambda, R_\mu)$ amenable to Dehn twist in special position always exists.

Consider a $4$-tuple $(\lambda, \mu, R_\lambda, R_\mu)$ which is amenable to Dehn twist in special position. Let $i(\lambda,\mu) = k$ and $K:= \{1, 2, ..., k\}$. We construct a curve in the isotopy class of $T_\lambda(\mu)$ which we call \textit{$T_\lambda(\mu)$ in special position w.r.t. the $4$-tuple $(\lambda, \mu, R_\lambda, R_\mu)$}. Start at any one of the components of $R_\lambda \cap R_\mu$ and label it as $A_1$. Since $\mu$ intersects $\lambda$ transversely, the arc $\mu_1$ of $\mu$ contained in $A_1$ which has its endpoints $X$ and $Y$ on boundary arcs of $R_\lambda$ is such that $X$ and $Y$ lie on distinct boundary component of $\partial R_\lambda$. We call the component of $\partial R_\lambda$ containing $X$ to be $\partial_{+}R_\lambda$ and the other component containing $Y$ to be $\partial_{-}R_\lambda$. Equip $A_1$ with the Euclidean metric such that it is a square in the $xy-$ plane. Two opposite sides of $A_1$ are formed from the arcs of $\partial R_\lambda$ and the two remaining sides formed from arcs of $\partial R_\mu$ and the $x$-axis lies along $\mu_1$ and the value of the $x$-coordinate increases from $X$ to $Y$. Orient $\mu_1$ from $X$ to $Y$. This induces an orientation on $\mu$. Next we pick $k$ distinct points $\{q_1, q_2,..., q_k\}$ in the interior of $\mu_1$ such that the $x$-coordinate of $q_i$ is greater than the $x$ coordinate of $q_j$ whenever $i>j$ and $i,j \in K$. For each $i \in K$, let $\lambda_i$ be a curve in $R_\lambda$ which is isotopic to $\lambda$ and passes through $q_i$. Further for each $i,j  \in K, i \neq j$ let $\lambda_i$ and $\lambda_j$ be disjoint. 

Orient $\lambda_1$ such that the $y$-coordinate on $\lambda_1$ increases when following this orientation in the disk $A_1$. Starting with $A_1$, label the subsequent disk components, $R_\lambda \cap R_\mu$, as $A_2, A_3, ..., A_k$, in the orientation of $\lambda_1$. For each $i \in K$, $A_i$ contains a unique arc of $\mu$ which we label as $\mu_i$. $\mu_i$ gets an induced orientation from $\mu$.  For each $i \in K$, equip $A_i$ with Euclidean metric and assume it to be a square in the $xy$-plane where $\mu_i$ lies along the $x$-axis with the $x$ coordinate increasing along the orientation of $\mu_i$. Assume $A_i$ to be positioned such that $\mu_i$ is the line segment joining the mid-points of the left and right sides of the square. In this orientation, call the component of $\partial R_\mu$ which appears above $\mu_i$ as $\partial_{+}R_\mu$ and the the component of $\partial R_\mu$ below $\mu_i$ as $\partial_{-}R_\mu$. However, note that the side of $A_i$ which is formed of the arcs of $\partial_{+}R_\lambda$ could either be to the right or to the left of this square. Accordingly, the side of $A_i$ which is formed of the arcs of $\partial_{-}R_\lambda$ could either be to the left or to the right of this square. For $i, j \in K$, by an isotopy inside $A_i$, we can assume that all the arcs of $\lambda_j$ in $A_i$ are straight lines.

For each $i,j \in K$, let $u_{i,j} := A_i \cap \lambda_j \cap \partial_{+}R_\mu$ and $v_{i,j} := A_i \cap \lambda_j \cap \partial_{-}R_\mu$. Also for each $i \in K$ let the left end point of $\mu_i$ in the square $A_i$ be $v_{i,0}$ and the right end point of $\mu_i$ in the square $A_i$ be $u_{i,k+1}$.  
Construct the Dehn twist of $\mu$ about $\lambda$ as follows: For each $j \in K \cup \{0\}$ draw line segments, $\theta_{i,j}$, connecting $v_{i,j}$ to $u_{i,j+1}$. $T_\lambda(\mu)$ is the curve $$((\mu \cup (\cup_{i \in K}\lambda_i)) \cap (S \setminus (\cup_{i \in K}A_i)) \cup (\cup_{i,j \in K}\theta_{i,j}).$$ The schematic, figure \ref{fig:DiskOfTransformation}, shows $A_i$ before and after this transformation. In the complement of $A_i$'s the transformation described above does not disturb the curves $\lambda_i$'s and $\mu$. In \cite{FM}, an algorithm to obtain the Dehn twist, $T_\lambda(\mu)$ has been described such that the curves in the discs of transformation are as in figure \ref{fig:DehnTwistSurgery}. The line segments in figure \ref{fig:DiskOfTransformation} are isotopic to the corresponding curves in \ref{fig:DehnTwistSurgery} which shows that the above transformation indeed results in $T_\lambda(\mu)$. When $T_\lambda(\mu)$ is constructed as above and as shown in figure \ref{fig:DiskOfTransformation}, we say that \textit{$T_\lambda(\mu)$ is in special position w.r.t. $\lambda$ and $\mu$}. We call the $k$ copies of $\lambda$, $\lambda_i$, $i \in K$, and $\mu$ to be the \textit {scaffolding for $T_\lambda(\mu)$}.  We call the Euclidean disks $A_i$, $i \in K$, along with the line segments $\theta_{i,j}$'s for $j \in K$ to be the \textit{disks of transformation for $T_\lambda(\mu)$}. The points $u_{i,j}$'s, $v_{i,j}$'s, $u_{i,k+1}$ and $v_{i,0}$ for $i, j \in K$ shall hold their meaning as defined in the context of the disks of transformations. So, using these phrases, when $T_\lambda(\mu)$ is in special position w.r.t. $\lambda$ and $\mu$, the scaffolding of $T_\lambda(\mu)$ remains unchanged outside its disks of transformation. Inside the disks of transformation for $T_\lambda(\mu)$, the schematic in figure \ref{fig:DiskOfTransformation} describes the changes to its scaffolding.

\begin{figure}
\centering
\includegraphics[scale=0.62]{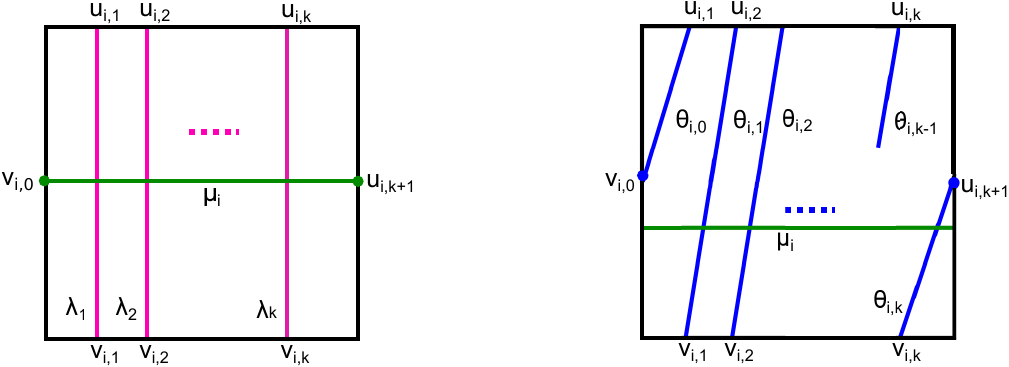}
\caption{Disk of transformation before (figure on the left) and after (figure on the right) the Dehn twist}
\label{fig:DiskOfTransformation}
\end{figure}

\begin{figure}
\centering
\includegraphics[scale=0.62]{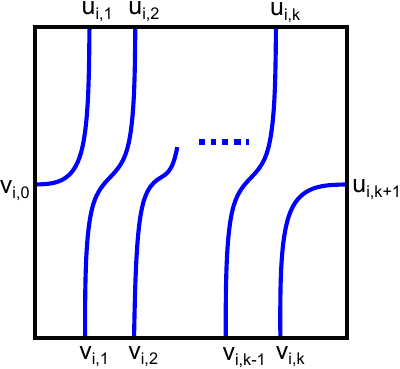}
\caption{Surgery of the curves to obtain $T_\lambda(\mu)$}
\label{fig:DehnTwistSurgery}
\end{figure}

\section{Distance 4 curves in $\mathcal{C}(S_{g \geq 2})$}
\label{section:main}

\begin{thm}\label{thm:main} Let $S$ be a surface of genus $g \geq 2$. Let $\alpha$ and $\gamma$ be two curves on $S$ with $d(\alpha, \gamma) = 3$. Then, $d(\tga, \alpha) = 4$.
\end{thm}

\begin{proof} Let $\nu_0$, $\nu_1$, $\nu_2$, $\nu_3$ be a geodesic from the vertex $\nu_0$ corresponding to $\alpha$ to the vertex $\nu_3$ corresponding to $\gamma$ in $\mathcal{C}(S)$. Let $T_\gamma(\nu_0)$ be the vertex in $\mathcal{C}(S)$ corresponding to $\tga$. The existence of the path $T_\gamma(\nu_0)$, $T_\gamma(\nu_1)$, $T_\gamma(\nu_2)$ = $\nu_2$, $\nu_1$, $\nu_0$ gives that $d(\tga,\alpha) \leq 4$. We prove that $d(\tga,\alpha) \geq 4$ by using Theorem \ref{thm:dist_geq_4} with $\kappa = T_\gamma(\alpha)$ and $\omega = \alpha$, hence showing that $d(\tga, \alpha) = 4$.

\begin{clm}\label{clm:fill}
$\tga$ and $\alpha$ fill $S$. 
\end{clm}

\textit{Proof of claim \ref{clm:fill}}:
Let $i(\gamma, \alpha)=k$, $K:= \{1, 2, ..., k\}$, $K_{-1}:= \{1, 2, ..., k-1\}$ and $K_{2-2g}:= \{1, 2, ..., k+2-2g\}$ . We refer to section \ref{section:special_position} for the terminology used here. Since $\alpha$ and $\gamma$ fill $S$, there is a $4$-tuple $(\alpha, \gamma, R_\alpha, R_\gamma)$ which is amenable to Dehn twist in special position. Let $\tga$ be in special position w.r.t to $\alpha$ and $\gamma$. We denote the disks of transformation of $\tga$ by $A_i$, for $i \in K$. By an isotopy we assume the curve $\alpha$ to be disjoint from $\tga \setminus A_i$ for $i \in K$ and in each $A_i$ we further assume the arc $\alpha_i:= \alpha \cap A_i$ to be a straight line segment below the segment connecting $v_{i,0}$ and $u_{i,k+1}$. 

\begin{figure}
\centering
\includegraphics[scale=0.55]{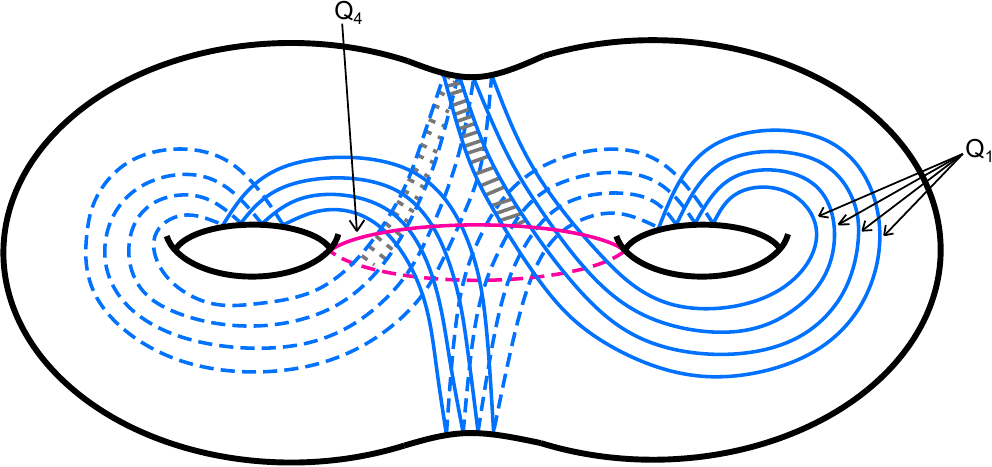}
\caption{The scaffolding for $T_{Q_1}(Q_4)$, where $Q_1$ and $Q_4$ are from example in figure \ref{fig_gen_2_dist_3_eg} and the shaded region is a rectangle of the scaffolding}
\label{fig:ScaffoldingOnGen2}
\end{figure}

For $i \in K$, let $g_i$ along with $\alpha$ be the scaffolding for $\tga$. For $j \in K_{-1}$, one of the components of $S \setminus \{g_j, g_{j+1}\}$ is an annulus, $G_j$. Any component of $G_j \setminus \alpha$ is a $4$-gon which we call as \textit{a rectangle of the scaffolding for $\tga$}. Figure \ref{fig:ScaffoldingOnGen2} shows an example of such a rectangle of the scaffolding. The disks $A_i$, $i \in K$, further divide each rectangle of the scaffolding into three components. There is a unique $i \in K$ such that $A_i$ and $A_{i+1}$ intersect a given rectangle of the scaffolding. Denote a rectangle of the scaffolding formed out of $G_j$ with its arcs of $\alpha$ lying in $A_i$ and $A_{i+1}$ by $B_{i,j}$. Denote the sub-rectangles $B_{i,j} \cap A_{i}$, by $C'_{i,j}$ and $B_{i,j} \cap A_{i+1}$, by $C''_{i+1,j}$. Also let $B'_{i,j} := B_{i,j} \setminus (C'_{i,j} \cup C''_{i+1,j})$.  Let $$B = \cup_{i=1}^k \cup_{j=1}^{k-1} B_{i,j}. $$ $S \setminus (\alpha \cup \gamma)$ has $k+2-2g$ disk components by Euler characteristic considerations. If $F_p$ is a disk component of $S \setminus (\alpha \cup \gamma)$, for some $p \in K_{2-2g}$, then $F'_p := F_p \setminus B$ is a single disk as $B$ intersects any $F_p$ only in disks which contain a boundary arc of $F_p$, namely arcs of $\gamma$. The components of $S \setminus (\alpha \cup g_1 \cup \dots \cup g_k)$ comprise of $k(k-1)$ rectangles of the scaffolding for $\tga$, namely $B_{i,j}$ where $i \in K$, $j \in K_{-1}$, and $k+2-2g$ even sided polygonal discs, namely $F'_p$, where $p \in K_{2-2g}$. Let $F''_p$ denote $F'_p \setminus R_\alpha$ for $p \in K_{2-2g}$.

\begin{figure}
\centering
\includegraphics[scale=0.62]{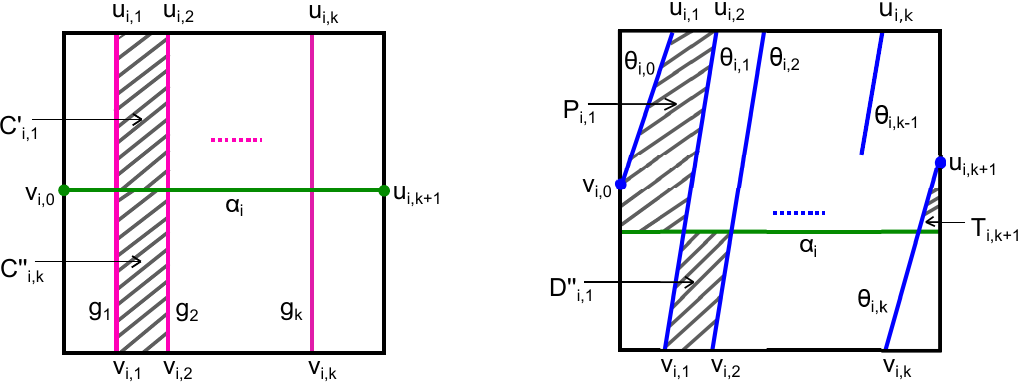}
\caption{The disk of transformation for $\tga$ : the figure on the left shows the portion of the scaffolding for $\tga$; the figure on the right shows the pentagon $P_{i,1}$, the triangle $T_{i,k+1}$ and the parallelograms formed due to $\alpha_i$ and $\tga$}
\label{fig:ShadedTransfDisk}
\end{figure}

For each $j \in K$ let $w_{i,j} := \theta_{i,j} \cap \alpha_i$. For each $i \in K$ and $j \in K_{-1}$, let $D''_{i,j}$ be the parallelogram with vertices $v_{i,j}, v_{i,j+1}, w_{i,j}$ and $w_{i,j+1}$ and $D'_{i,j+1}$ be the parallelogram with vertices $w_{i,j}, w_{i,j+1}$, $u_{i,j+1}, u_{i,j+2}$.  In each disk $A_i$, for $i \in K$, there is a pentagon, $P_{i,1}$, which is above $\alpha_i$ and bounded by the lines $\theta_{i,0}$, $\partial R_\gamma$, $\alpha_i$, $\theta_{i,1}$ and the line segment of $\partial_{+} R_\alpha$ between $u_{i,1}$ and $u_{i,2}$. Likewise, in each disk $A_i$, for $i \in K$, there is a triangle, $T_{i,k+1}$, which is bounded by the lines $\alpha_i$, $\theta_{i,k}$ and $\partial R_\gamma$. Figure \ref{fig:ShadedTransfDisk} shows a schematic before and after the transformation to the disk $A_i$; the figure to the left shows the rectangles $C'_{i,1}$ and $C''_{i,k}$ and the figure on the right shows $P_{i,1}$ and $T_{i,k+1}$.

Figure \ref{fig:DTInRAlpha} shows a schematic of $R_\alpha$ before and after the transformation to the scaffolding of $\tga$. The shaded region in the figure on the left shows $C'_{i,j}$ and $C''_{i,j-1}$ for some indices $i,j$. The shaded region in the figure on the right shows $D'_{i,j}$ and $D''_{i,j-1}$ for some indices $i,j$.

\begin{figure}
\centering
\includegraphics[scale=0.6]{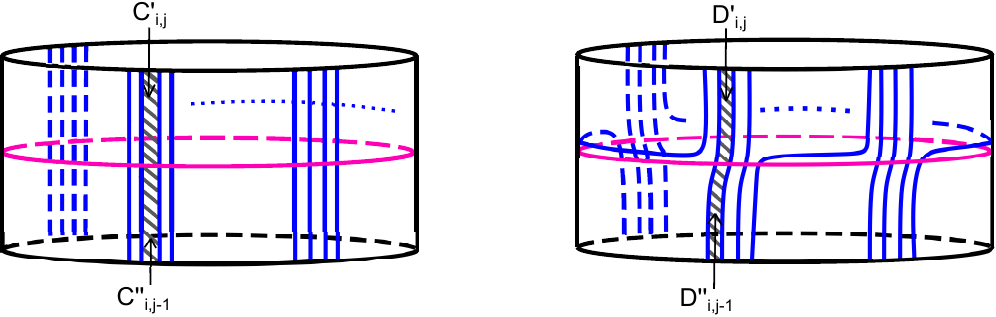}
\caption{A schematic of $R_\alpha$ (figure on the left) and after (figure on the right) the Dehn twist}
\label{fig:DTInRAlpha}
\end{figure}

For $i \in K$, note that all the disks $A_i$, occur in some sequence in the annulus $R_\alpha$ when moving along $\alpha$. So, a disk $A_i$ is connected to some disk $A_j$ on the left and to some other disk $A_p$ on the right by a single arc of $\alpha \setminus R_\gamma$, for some distinct indices $i, j, p \in K$. The schematic for two disks $A_i$ and $A_j$, for some $i, j \in K$, which are connected via a single arc of $\alpha \setminus R_\gamma$ and an arc of $\tga \setminus R_\gamma$ is as shown in the figure \ref{fig:AdjDOT}. Note that this schematic is generic since for every $j \in K$, there is a distinct $i \in K$ such that $A_j$ occurs to the left of $A_i$, in the sense mentioned above. 

\begin{figure}
\centering
\includegraphics[scale=0.63]{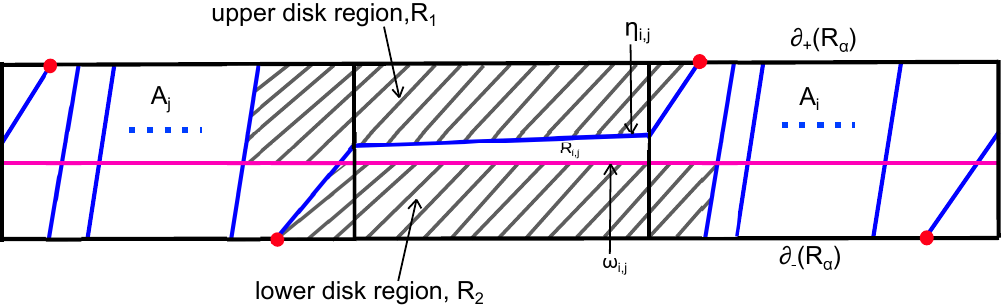}
\caption{Two adjacent disks of transformation in $R_\alpha$}
\label{fig:AdjDOT}
\end{figure}

Figure \ref{fig:AdjDOT} is a schematic of a portion of figure \ref{fig:DTInRAlpha}. 
In any of the cases,\textit{viz.} $\partial_+ R_\gamma$ and $\partial_+ R_\gamma$ face each other, $\partial_+ R_\gamma$ and $\partial_- R_\gamma$ face each other or $\partial_- R_\gamma$ and $\partial_- R_\gamma$ face each other, 
In this schematic, we see that the pentagon $P_{i,1}$ of the disk $A_i$ is connected to the triangle $T_{j,k+1}$ of $A_j$ via an arc of $\alpha \setminus R_\gamma$, $\omega_{i,j}$, and an arc of $\tga$, $\eta_{i,j}$. The disk, $R_{i,j}$ outside $R_\gamma$ bounded by $\omega_{i,j}$, $\eta_{i,j}$ and two arcs of $\partial R_\gamma$, will be called a \textit{conduit}. Equip the conduit with the Euclidean metric and assume that $R_{i,j}$ is a rectangle with two opposite sides $\omega_{i,j}$ and $\eta_{i,j}$. Now $P_{i,1} \cup R_{i,j} \cup T_{j,k+1}$ is a $4$-gon bounded by four arcs \textit{viz.} (i) $\theta_{i,0} \cup \eta_{i,j} \cup \theta_{j,k}$, (ii) $\alpha_j \cup \omega_{i,j} \cup \alpha_i$, (iii) $\theta_{i,1}$ and (iv) the arc of $\partial_{+}R_\alpha$ between $u_{i,1}$ and $u_{i,2}$. This protracted $4$-gon will be denoted by $D'_{i,1}$.

Let $S' = S \setminus R_\alpha$. The components of $S \setminus (\alpha \cup \tga)$  are the components of $S' \setminus  \tga$ and the components of $R_\alpha \setminus (\alpha \cup \tga)$ glued at the boundary of $R_\alpha$. Since the changes to the scaffolding of $\tga$ is restricted to $R_\alpha$, the components of $S' \setminus \tga$ are precisely the disc components of $S' \setminus (g_1 \cup \dots \cup g_k)$.

The components of $S' \setminus (g_1 \cup \dots \cup g_k)$ are $B'_{i,j}$, $i \in \{1, 2 \dots k\}$, $j \in K_{-1}$, along with disks $F''_p$, $p \in K_{2-2g}$, as explained above. The components of $R_\alpha \setminus (\alpha \cup \tga)$ will be examined using the schematic figure \ref{fig:AdjDOT} of a portion of $R_\alpha$. There are four kinds of regions in $R_\alpha$. The upper disk regions, like $R_1$ in the schematic figure \ref{fig:AdjDOT}, the lower disk regions, like $R_2$ in the schematic figure \ref{fig:AdjDOT}, and the disks $D'_{i,j}$, $D''_{i,j}$, $i \in K$, $j \in K_{-1}$. Figure \ref{fig:AdjDOT} shows how the upper and lower disk regions are glued to disks $F''_p$ for $p \in K_{2-2g}$. For each $p \in K_{2-2g}$, after gluing the lower disk regions and the upper disk regions to the respective disks $F''_p$, we get disks which we denote by $F'''_p$. We know that $F'''_p$ is a disk because the upper and the lower disk regions are disjoint, except for the points $w_{i,j}$ on the boundary and share a single arc of $\partial R_\alpha$ with a unique $F''_p$. For each $p \in K_{2-2g}$, we call $F'''_p$ to be the \textit{modified disk} corresponding to the initial disk $F_p$.

For each $i \in K$ and $j \in K_{-1}$, the line segment of $\partial_{+}R_\alpha$ between $u_{i,j}$ $u_{i,j+1}$ is the common boundary of $C'_{i,j}$ and $D'_{i,j}$. Likewise, for each such $i, j$, the line segment of $\partial_{-}R_\alpha$ between $v_{i,j}$ $v_{i,j+1}$ is the common boundary of $C''_{i,j}$ and $D''_{i,j}$. So, for such $i, j$, when considering the components of $S \setminus  (\alpha \cup g_1 \cup \dots \cup g_k)$ the rectangular core $B'_{i,j}$ is connected to $C'_{i,j}$ along the boundary segment $u_{i,j}$ $u_{i,j+1}$ and to $C''_{i+1,j}$ along the boundary segment $v_{i+1,j}$ $v_{i+1,j+1}$, whereas when considering the components of $S \setminus (\alpha \cup \tga)$, the rectangular core $B'_{i,j}$ is connected to $D'_{i,j}$ along the boundary segment $u_{i,j}$ $u_{i,j+1}$ and $D''_{i+1,j}$ along the boundary segment $v_{i+1,j}$ $v_{i+1,j+1}$. So the rectangles of the scaffolding for $\tga$, $B_{i,j}$, which are components of $S \setminus  (\alpha \cup g_1 \cup \dots \cup g_k)$, after the transformation in the disks of transformation for $\tga$ result in disks $E_{i,j}:= B'_{i,j} \cup D'_{i,j} \cup D''_{i+1,j}$ which now are components of $S \setminus (\alpha \cup \tga)$. For each $p \in K_{2-2g}$, $F'''_p$ is a disk as seen earlier. The components of $S \setminus (\alpha \cup \tga)$ are precisely the disks $F'''_p$ and $E_{i,j}$ where $p \in K_{2-2g}$, $i \in K$ and $j \in K_{-1}$. This proves that the components of $S \setminus (\alpha \cup \tga)$ are all disks and hence proving Claim \ref{clm:fill}.

The components of $R_{\gamma} \setminus \tga$ are disks and their boundary consists of two arc segments of $\tga$ and one each of $\partial_{+}R_\gamma$ and $\partial_{-}R_\gamma$. We call these disks as \textit{rectangular tracks}. The word tracks derives its motivation from how these tracks appear in $R_\gamma$. Figure \ref{fig:Tracks} shows $R_{\gamma}$ and rectangular tracks inside $R_\gamma$.

\begin{figure}
\centering
\includegraphics[scale=0.6]{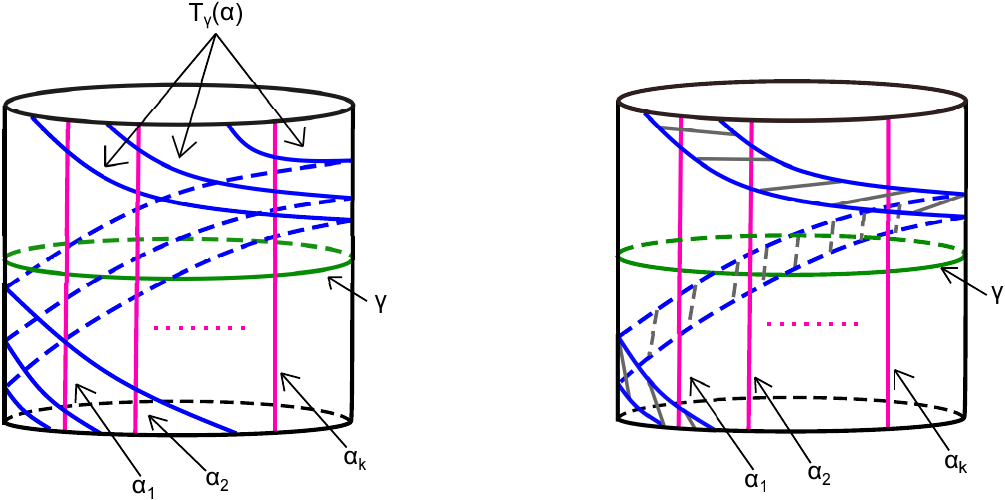}
\caption{The rectangular tracks shown inside the annulus $R_\gamma$}
\label{fig:Tracks}
\end{figure}

Since $i(\alpha, \gamma) = k$, there are $k$ components of $\alpha \cap R_\gamma$. Every component of $\alpha \setminus \tga$ is either contained in $R_\gamma$ or, has a sub-arc which is contained in $R_\gamma$. For any $i \in K$, $\alpha_i$ intersects the rectangular tracks. 

Let $i_0\in K$. In the schematic figure \ref{fig:ShadedTransfDisk}, $A_{i_0}$ has exactly $k+1$ arcs of $\tga$. Call $\theta_{{i_0},0}$ to be the \textit{leftmost arc of $A_{i_0}$} and $\theta_{{i_0},k}$ to be the \textit{rightmost arc of $A_{i_0}$}. Let us consider one component of $\tga \cap R_\gamma$, call it $\rho_{i_0}$, which intersects $A_{i_0}$ in its leftmost arc. This $\rho_{i_0}$ intersects $A_{i_0}$ precisely in the arcs  $\theta_{{i_0},0}$ and $\theta_{{i_0},k}$ and it intersects $A_j$ for every $j \in K \setminus \{i_0\}$ in the arcs $\theta_{j,m}$ where $m=(j-i_0)(mod\; k)$. This is easily seen from the construction of $\tga$ in special position w.r.t. $\alpha$ and $\gamma$. From this discussion it is clear that $\rho_{i_0}$ intersects each $\alpha_j$, for $j \in K$, exactly once. It is also clear that, for $j \in K$, the points of $\rho_{i_0} \cap \alpha_j$ lie on $\rho_{i_0}$ in the order $\alpha_{{i_0}+1}, ..., \alpha_k, \alpha_1, ..., \alpha_{{i_0}-1}, \alpha_{i_0}$ when  $\rho_{i_0}$ is traversed from $\partial_{+}R_\gamma$ to $\partial_{-}R_\gamma$. We now consider two arc components, $\rho_{i_0}$ and $\rho_{i_0+1}$, of $T_{\gamma}(\alpha) \cap R_{\gamma}$ and the rectangular track, $T_{i_0}$, which is enclosed by these two components in $R_\gamma$. We equip this rectangular tracks $T_{i_0}$ with the Euclidean metric so that the boundary arcs $\rho_{i_0}$, $\rho_{i_0+1}$, and the arcs of $T_{i_0} \cap \partial R_\gamma$ are all straight lines and so that $T_{i_0}$ is a rectangle. We refer to $T_{i_0} \cap \partial_{+}R_\gamma$ as the left end of the rectangle and $T_{i_0} \cap \partial_{-}R_\gamma$ as the right end of this rectangular track. We can draw the arcs of $\alpha_j$, for $j \in K$, as straight line segments in the rectangular tracks $T_{i_0}$. Figure \ref{03} shows a schematic of $T_i$ where $i \in K$.

\begin{figure}
\centering
\includegraphics[scale=0.65]{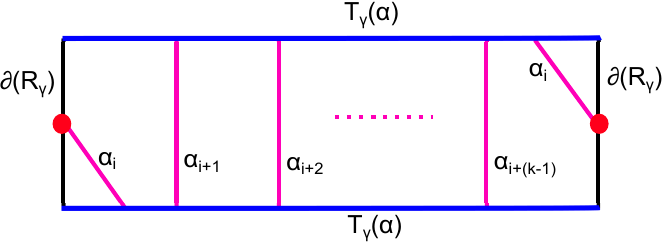}
\caption{A rectangular track $T_i$ along with arcs of $\alpha_i$ in it}
\label{03}
\end{figure}

From this schematic, at both the left and right end of this rectangular track $T_i$, $a_i$ is a common boundary to a triangle and a pentagon. We call $\alpha_i$ as the \textit{starting arc} of this rectangular track $T_i$.

Figure \ref{08} shows two possible schematics when $A_i$ is pictured in $R_\gamma$.

\begin{figure}
\centering
\includegraphics[scale=0.7]{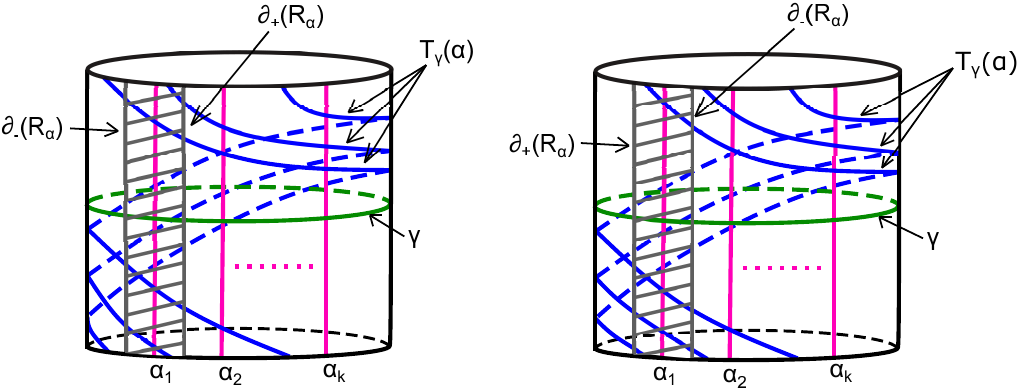}
\caption{$A_i$ shown inside $R_\gamma$ in the two possible ways : the figure on the left shows $\alpha_i$ oriented from top to bottom; the figure on the right shows $\alpha_i$ oriented from bottom to top}
\label{08}
\end{figure}

For any of the two possible cases observed in figure \ref{08}, a portion of one of the two pentagons of $T_i$ appears in the $A_i$ which is between $\alpha_i$ and $\partial_{+}R_\alpha$, where $\alpha_i$ is the starting arc of this track. We call this pentagon \textit{the upper pentagon of the rectangular track $T_i$}, owing to the viewpoint that $\partial_{+}R_\alpha$ is the upper boundary of $R_\alpha$. A portion of the other pentagon of $T_i$ appears in $A_i$ which is between $\alpha_i$ and $\partial_{-}R_\alpha$. We call this pentagon \textit{the lower pentagon of the rectangular track}. Likewise, we define the \textit{upper triangle} and the \textit{lower triangle} of a rectangular track $T_i$. 

Let $\oga \in \Gamma$ as in the statement of the Theorem \ref{thm:dist_geq_4}. We prove that $d(\oga, \alpha) \geq 3$ by showing that $\oga$ and $\alpha$ fill $S$. By Theorem \ref{thm:dist_geq_4}, this will imply that $d(\tga, \alpha) \geq 4$. 

It can be observed that $i(\oga \cap \alpha) \neq 0$ because if $\oga$ is disjoint from both $\alpha$ and $\tga$ then $\oga$ is non-essential as it will lie completely in one of the disc components of $S \setminus (\tga \cup \alpha)$. Since $d(\oga, \gamma) \geq d(\tga, \gamma) - d(\oga, \tga) = d(\tga, T_\gamma(\gamma)) - d(\oga, \tga) = 3 - 1 = 2$, we also conclude that $i(\oga \cap \gamma) \neq 0$. Since $\oga$ intersects $\gamma$, it intersects $R_\gamma$. It cannot be completely contained in $R_\gamma$ because every simple closed curve contained in an annulus bounds a disk or is isotopic to the core curve of the annulus. Since neither of these is true, it follows that that $\oga$ intersects $R_\gamma$ in arcs. Since $i(\oga, \tga) = \phi$, each component of $\oga \cap R_\gamma$ has to be completely contained in one of the rectangular tracks described by $\tga$. Such a component arc of $\oga$ could either be boundary reducible or essential in $R_{\gamma}$.

We consider an isotopy $I_1$ of $\oga$, as follows: In the case that a component arc of $\oga$ in $R_{\gamma}$ is boundary reducible in $R_{\gamma}$, we can perform the boundary reduction of $\oga$ preserving its minimal intersection position with $\alpha$ and $\tga$. This is possible because an arc of $\oga$ which is boundary reducible in $R_{\gamma}$ and is contained in the disk $T_i$ will bound a bigon with one boundary arc of $R_{\gamma}$ in $T_i$. Also, since $\oga$ was already in minimal intersection position with $\alpha$, it does not bound bigons with the arcs $\alpha_j$ inside $T_i$. Call the isotopy of $\oga$ which reduces all the boundary-reducible arcs of  $\oga \cap R_{\gamma}$ as $I_1$. After the isotopy $I_1$, we can assume that all the arcs of $\oga$ in $R_{\gamma}$ are essential. We know that there is at-least one component of $\oga \cap R_\gamma$ which is an essential arc of $R_\gamma$ as $\oga$ cannot be disjoint from $R_\gamma$. By the hypothesis that $i(\oga, b) \leq 1$ for $b \subset \alpha \setminus \tga$ each rectangular track can contain at-most one component of $\oga \cap R_\gamma$. 

Next, we describe an isotopy $I_2$ of $\oga$ such that all the points of $\oga \cap \alpha$ will lie inside $R_\gamma$ and so that no new boundary reducible arc components of $\oga \cap R_\gamma$ are introduced and $\oga$'s minimal intersection position with $\alpha$  and $\tga$ is retained. To this end, suppose that a point of $\oga \cap \alpha$ lies outside $R_\gamma$.

Following the construction of the disk $D'_{i,1}$ described above using figure \ref{fig:AdjDOT}, we see that the upper pentagon of the rectangular track $T_i$ is connected to the upper triangle of the rectangular track $T_j$ via a conduit $R_{i,j}$ where $i, j \in K$ are such that $A_j$ is to the left of $A_i$ in $R_\alpha$ as in schematic \ref{fig:AdjDOT}.

If a point of $\oga \cap \alpha$, $x_0$, lies outside $R_\gamma$, then it has to lie on $\omega_{i,j}$ for some $i$ and $j$ such that $i,j \in K, i \neq j$. We now refer to the dotted line in figure \ref{011c}. Since the intersection of $\oga$ and $\alpha$ is transverse, an arc of $\oga$, call it $\delta$ lies on the two sides of the conduit $R_{i,j}$, one inside and one outside $R_{i,j}$. The endpoint $P$ of the arc $\delta$ inside $R_{i,j}$ is also the endpoint of some other arc of $\oga$ as $\oga$ is a closed curve. If $P$ connects to an arc of $\oga$ lying in the upper triangular region of the track $T_j$, then an essential arc $\delta_1$ of $\oga \cap R_\gamma$ lies in $T_j$ with its endpoint $Q$ on $\partial R_\gamma$ in the upper triangle of $T_j$ so that $\delta$, the arc $PQ$ and $\delta_1$ together form a bigon with $\alpha$ contradicting the minimal intersection position of $\oga$ with $\alpha$. So, $P$ connects to an arc of $\oga$ in the upper pentagon in the track $T_i$ as is the dotted line in figure \ref{011c}. Consider an isotopy $I_2$ which slides the point $x_0$ onto $\alpha_i$. The image of the arc component of $\oga \cap R_\gamma$ which is in $T_i$, under $I_2$ has its endpoint in the lower triangle of $T_i$ and the image of $x_0$ lies in $R_\gamma$. A schematic for this isotopy $I_2$ is shown in figure \ref{011c}.

\begin{figure}
\centering
\includegraphics[scale=0.7]{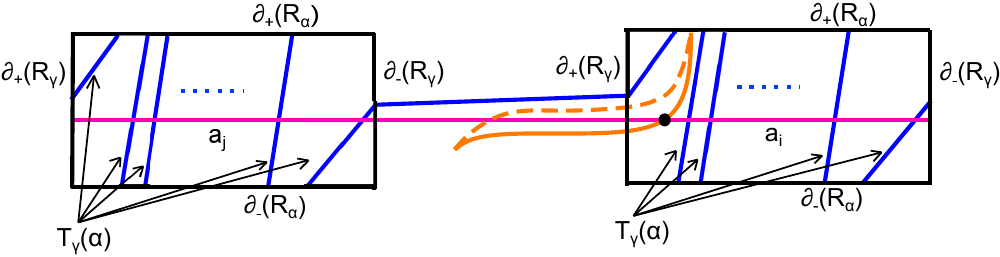}
\caption{The isotopy $I_2$ moving points of $\oga \cap \alpha$ into $R_\gamma$}
\label{011c}
\end{figure}

After finitely many such isotopies, we can now assume that all the points of $\oga \cap \alpha$ lie inside $R_\gamma$. Now consider an isotopy $I_3$ of $\oga$ as follows: If any of the components of $\oga \cap R_\gamma$ has its endpoint on the boundary of the upper triangle of $T_j$, for some $j \in K$ , then by the above discussion, $\oga$ cannot intersect $\omega_{i,j}$ or $\eta_{i,j}$, for some $i \in K$ such that the arcs of $T_i$ and $T_j$ forms the opposite sides of a conduit $R_{i,j}$. So $\oga \cap R_{i,j}$ is an arc $MN$ which has its endpoints $M \in T_j$ and $N \in T_i$ on $\partial R_\gamma$. Further, since $\oga$ is a closed curve, $\oga \cap T_i$ is an arc with its endpoint as $N$ such that $N$ necessarily lies in the upper pentagon of $T_i$. Conversely, if any of the components of $\oga \cap R_\gamma$ has its endpoint, $z_0$, on the boundary of the upper pentagon of $T_i$, then it should be connected to an arc, $g$, of $\oga$ in the conduit $R_{i,j}$. Note that the endpoints, $z_0, z_0'$ of $g$ are on $\partial R_\gamma$. There exists an arc component of $\oga \cap R_\gamma$ lying in $T_j$ such that $z_0'$ is on the boundary of the upper triangle of $T_j$, as the dotted line in figure \ref{011e} shows. If any such arc $g$ of $\oga$ exists, consider an isotopy, $I_3$, of $g$ such that the image, $I_3(g)$, lies outside $R_{i,j}$. A schematic of this is figure \ref{011e}. 

\begin{figure}
\centering
\includegraphics[scale=0.7]{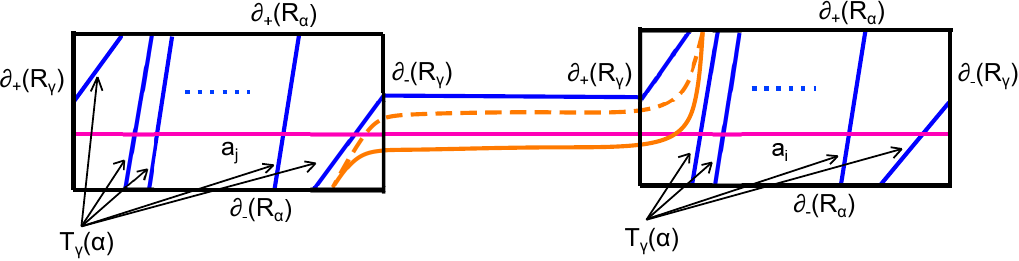}
\caption{A schematic showing the normalization move, the isotopy $I_3$}
\label{011e}
\end{figure}

The component of $\oga \cap R_\gamma$ in $T_j$ now has an endpoint on the boundary of the lower pentagon of $T_j$ and the component of $\oga \cap R_\gamma$ in $T_i$ has an endpoint on the boundary of the lower triangle of $T_i$. Also the image of $\oga \cap \alpha$ under $I_3$ moves a point of $\oga \cap \alpha$ from the boundary of the upper traingle of $T_j$ to the boundary of the lower pentagon of $T_i$. We call $I_3$ to be a \textit{normalization} move on $\oga$. After finitely many normalization moves performed on $\oga$, wherever applicable, we can assume that every component of $\oga \cap R_\gamma$ is contained in a rectangular track $T_i$ for some $i \in K$ such that the endpoints of that component lie on the boundary of the lower triangle and the lower pentagon of $T_i$. So a schematic of every component of $\oga \cap R_\gamma$ inside $T_i$ is as in figure \ref{012}.

\begin{figure}
\centering
\includegraphics[scale=0.7]{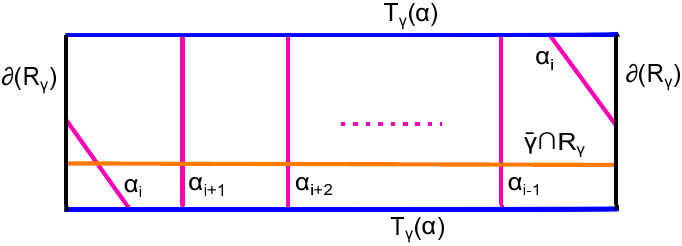}
\caption{The portion of $\oga$ in rectified position inside $T_i$}
\label{012}
\end{figure}

After these isotopies $I_1, I_2, I_3$ of $\oga$, we say that $\oga$ is in a \textit{rectified position}. We now prove that $\oga$ in rectified position and $\alpha$ fill $S$. From now on we assume that $\oga$ is in a rectified position.

For $i \in K$, let $H_i$ be the rectangular component of $R_\gamma \setminus (\cup_{i \in K}\alpha_i)$ containing the arcs $a_i$ and $a_{i+1}$ on its boundary. Each of these $H_i$ contains a unique segment, $\gamma_i$, of the core curve $\gamma$. The schematic \ref{09} shows $H_1$ and $\gamma_1$ for instance.

\begin{figure}
\centering
\includegraphics[scale=0.62]{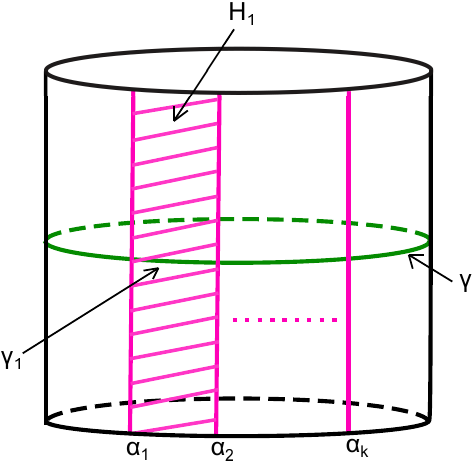}
\caption{Schematic showing $H_1$ and $\gamma_1$ in $R_\gamma$}
\label{09}
\end{figure}

We say that an arc, $g$ of $\oga$ \textit{covers $\gamma_i$} if $g \subset H_i$ has its end points on $\alpha_i$ and $\alpha_{i+1}$ and $g$ is isotopic in $H_i$ to $\gamma_i$ through arcs whose end points stay on $\alpha_i$, $\alpha_{i+1}$. Since $\gamma$ and $\alpha$ form a filling pair, the set of essential arcs, $\{\gamma_1, \dots, \gamma_k\}$ fill $S \setminus \alpha$. It follows that $\oga$ fills $S$ along with $\alpha$ if segments of $\oga \setminus \alpha$ cover $\gamma_i$ for all $i$ with $i \in K$.

Since $\oga$ is in rectified position, each component of $\oga \cap R_\gamma$ already covers all $\gamma_i$ except one as in figure \ref{012}. More precisely, if a component of $\oga \cap R_\gamma$ is in a rectangular track $T_i$, then $\oga$ covers every $\gamma_j$ where $j$ is such that $1 \leq j \leq k$ and $j \neq i-1$. So, if $\oga \cap R_\gamma$ has two distinct components, then each component has to lie in $T_i$ for distinct $i$ and hence $\oga$ covers $\gamma_j$ for $j \in \{1, 2, \dots k\}$. We conclude that $\oga$ and $\alpha$ fill $S$ in this case. Now it remains to show that if there is a single component of $\oga \cap R_\gamma$, which is an essential arc of $R_\gamma$ and is contained in some rectangular track $T_i$, then $\oga$ and $\alpha$ fill $S$. As in the previous case, $\oga$ covers every $\gamma_j$ where $j$ is such that $1 \leq j \leq k$ and $j \neq i-1$. The components of $S \setminus (\alpha \cup _{1 \leq j \leq k, j \neq i-1} \gamma_j)$ will be disks except possibly one which could be a cylinder. This can be seen as follows. Since $\alpha$ and $\gamma$ fill $S$, the components of $S \setminus \alpha \cup _{1 \leq j \leq k} \gamma_j$ are disks. Each segment of $\gamma_j \setminus \alpha$ for $j \in \{1,2,...,k\}$ contributes to two distinct edges of a component $J_0$ or two separate components $J, J'$ of $S \setminus \alpha \cup _{1 \leq j \leq k} \gamma_j$.

Let $P_1:= \oga \cap \alpha_i$ and $P_2:= \oga \cap \alpha_{i-1}$ be points in $T_i$ which appear on the unique component of $\oga \cap R_\gamma$. Let $[P_1, P_2]$ represent the arc of $\oga$ in $R_\gamma$ with endpoints $P_1$ and $P_2$ and $\oga_1:= \oga \setminus [P_1,P_2]$. $\oga_1$ is contained in all the components of $S \setminus \alpha \cup _{1 \leq j \leq k, j \neq i-1} \gamma_j$ which contain the arcs $\alpha_{i-1}$ and $\alpha_i$ on their boundary. We know that there is at-least one such component because $\gamma_{i-1}$ is also such an arc which joins $\alpha_{i-1}$ to $\alpha_i$. If $\gamma_{i-1}$ is the boundary of $J, J'$, then it would have been an arc which connected $\alpha_{i-1}$ on one disk to $\alpha_i$ on another disk. Note that both $\alpha_i$ and $\alpha_{i-1}$ are also boundary arcs of both $J$ and $J'$. So, we would find $P_1$ on the disk containing $\alpha_i$ and $P_2$ on the disk containing $\alpha_{i-1}$. When we join $J$ and $J'$ along $\gamma_{i-1}$ we get a disk where $\oga_1$ is an arc from $P_1$ to $P_2$ intersecting $\gamma_{i-1}$. Cutting along $\oga_1$ still yields two different disks. The schematic, figure \ref{013} shows this situation.

\begin{figure}
\centering
\includegraphics[scale=0.65]{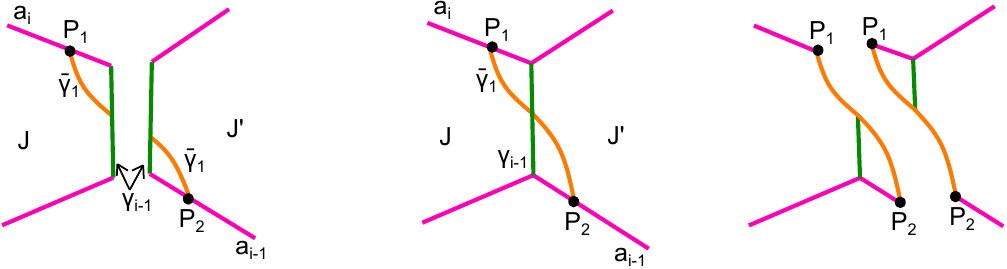}
\caption{The figure on the left shows disks $J$ and $J'$ formed by cutting along $\gamma_{i-1}$; the figure on the right shows the new disks formed when $J \cup J'$ are cut along $\oga_1$}
\label{013}
\end{figure}

If $\gamma_{i-1}$ were on the boundary of $J_0$ representing two edges of $J_0$ then it would have been an arc which connected $\alpha_{i-1}$  to $\alpha_i$. When we glue $J_0$ to itself along $\gamma_{i-1}$, we get a cylinder, $A$, where $\alpha_i$ and $\alpha_{i-1}$ will be arcs on different boundary components of $A$. So we would find $P_1$ and $P_2$ on distinct boundaries of $A$ and hence $\oga_1$ would be an essential arc on $A$. So cutting $A$ along this arc $\oga_1$ would yield a disk as shown in the schematic, figure \ref{014}.

\begin{figure}
\centering
\includegraphics[scale=0.65]{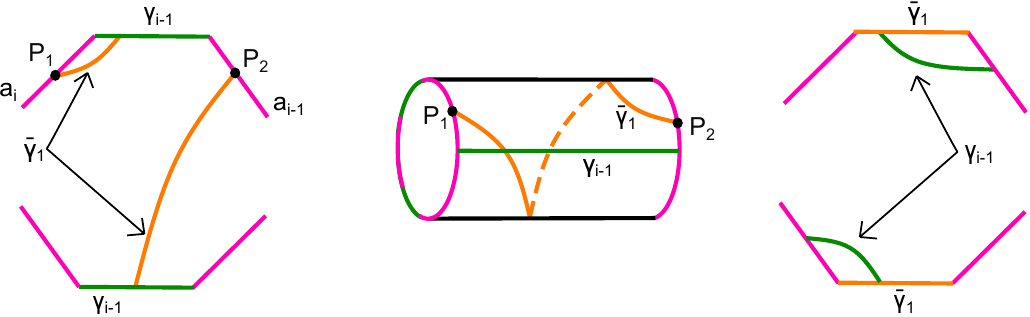}
\caption{The disk $J_0$ glued to itself along $\gamma_{i-1}$ and cut along $\oga_1$}
\label{014}
\end{figure}

In any case, we get disks by cutting $S \setminus \alpha$ along the arcs of $\oga \setminus \alpha$. So this proves the theorem.

\end{proof}

Claim \ref{clm:fill} of Theorem \ref{thm:main} doesn't use the hypothesis that $d(\alpha, \gamma) =3$. So we have the following:

\begin{cor}\label{cor:fill} If $\alpha$ and $\gamma$ are a pair of curves which fill $S$, then $\alpha$ and $\tga$ also fill $S$.
\end{cor}

\begin{cor}\label{cor:upper_bound} For a surface of genus $g \geq 3$, $i_{min}(g,4) \leq (2g-1)^2$.
\end{cor}
\begin{proof}
Aougab and Huang \cite{AH1} proved that $i_{min}(g,3) = 2g-1$ for $g \geq 3$. Now, on $S_g$, for $g \geq 3$, suppose that $\alpha$ and $\beta$ are two such minimally intersecting curves with $d(\alpha, \beta)=3$. Then $i(\alpha, T_\beta(\alpha)) = (2g-1)^2$ and by Theorem \ref{thm:main}, $d(\alpha, T_\beta(\alpha)) = 4$. So $i_{min}(g,4) \leq (2g-1)^2$.

\end{proof}

\section{An initially efficient geodesic}\label{ineff}

\begin{lem} If $\alpha = \nu_0, \nu_1, \nu_2, \nu_3 = \gamma$ is an initially efficient geodesic then so is $T_\gamma(\alpha), T_\gamma(\nu_1), \nu_2, \nu_1, \alpha$. \end{lem}

\begin{proof}
For $p \in K_{2-2g}$, let $F''_p$ be the components of $S \setminus \{\alpha, R_\gamma\}$ as in the proof of Theorem \ref{thm:main}. Since the geodesic $\alpha, \nu_1, \nu_2, \gamma$ is an initially efficient one, each segment of $\nu_1$ intersects every reference arc in $E_i$ at most twice. In particular, arcs of $\partial(R_\gamma)$ that form the edges of $E_i$ intersect $\nu_1$ at most twice. It follows from here that there are at the most two segments of $\nu_1$ in each rectangular track $T_i$ as defined in section \ref{section:main}. A schematic of this is shown in figure \ref{DT17}. Further, since the interior of a reference arc is disjoint from $\alpha \cup T_\gamma(\alpha)$, it is sufficient to check for the initial efficiency of the geodesic, $T_\gamma(\alpha), T_\gamma(\nu_1), \nu_2, \nu_1, \alpha$ in the modified disks $F'''_p$, abbreviated $F$, corresponding to $F_p$ , abbreviated $E$.

Since $E$ and $F$ are homeomorphic to a $2g$-gon. Without loss of generality assume $E$ and $F$ to be a regular Euclidean regular polygon with $2g$ sides. Starting at any segment of $\alpha$ in $E$, we label the edge as $\alpha_1$. Label the edges of $E$ in a clockwise direction, starting at $\alpha_1$ as $\gamma_1$, $\alpha _2$, $\gamma _2$, \dots , $\gamma _g$. Let $S' = S \setminus R_\gamma$. Since the components of $S' \setminus \{\alpha, \gamma\}$ and $S' \setminus \{\alpha, T_\gamma(\alpha)\}$ are the same, it follows that for every edge, $a_{j_0}$ in $F$ corresponding to $\alpha$, there exists a unique $i_0 \in \{1, \dots, g\}$ such that $a_{j_0} \subset \alpha_{i_0}$. Index the edges, $a_{j_0}$ of $F$ such that $j_0=i_0$. Label the edge of $\tga$ in $F$ between $a_i$ and $a_{i+1}$ as $t_i$. Let $\omega$ be a reference arc in $F$ with end points on $t_p$ and $t_q$ for some $p,q \in \{1, \dots, g \}$. Suppose to the contrary that $\omega \cap T_\gamma(\nu_1) \geq 3$. Then there exists three segments, $z_1$, $z_2$, $z_3$ of $T_\gamma(\nu_1)$ in $F$ such that $z_j \cap \omega \neq \phi$. For $j\in \{1,2,3\}$, let the end points of $z_j$ lie on $a_{j_1}$ and $a_{j_2}$. From our previous discussion on Dehn twist and figure \ref{DT18}, there exists arcs of $\nu_1$ in $E$ with end points on $\gamma_{j_1}$ and $\gamma_{j_2}$ for all $j \in \{1,2,3\}$. Consider a line segment, $\omega'$ in $E$ from an interior point of $a_p$ to an interior point of $a_q$. Then $\omega'$ is a reference arc for the triple, $\alpha$, $\nu_1$, $\gamma$ and $\omega' \cap \nu_1 \geq 3$. This contradicts that $\alpha$, $\nu_1$, $\nu_2$, $\gamma$ is an initially efficient geodesic. Hence, $\omega \cap T_\gamma(\nu_1) \leq 2$ for any choice of reference arc, $\omega$ for the triple $T_\gamma(\alpha)$, $T_\gamma(\nu_1)$, $\alpha$.

\begin{figure}
\centering
\includegraphics[scale=0.65]{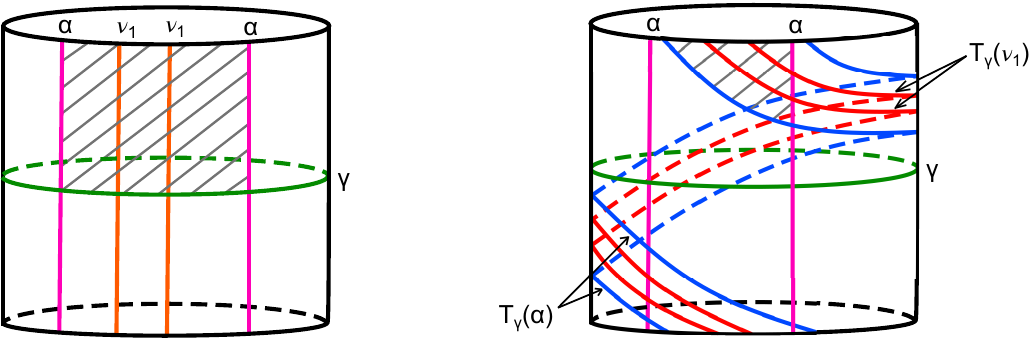}
\caption{There can be at-most two distinct segments of $T_\gamma(\nu_1)$ in any rectangular component of $S \setminus (\alpha \cup \tga)$ in $R_\gamma$}
\label{DT17}
\end{figure} 

\begin{figure}
\centering
\includegraphics[scale=0.7]{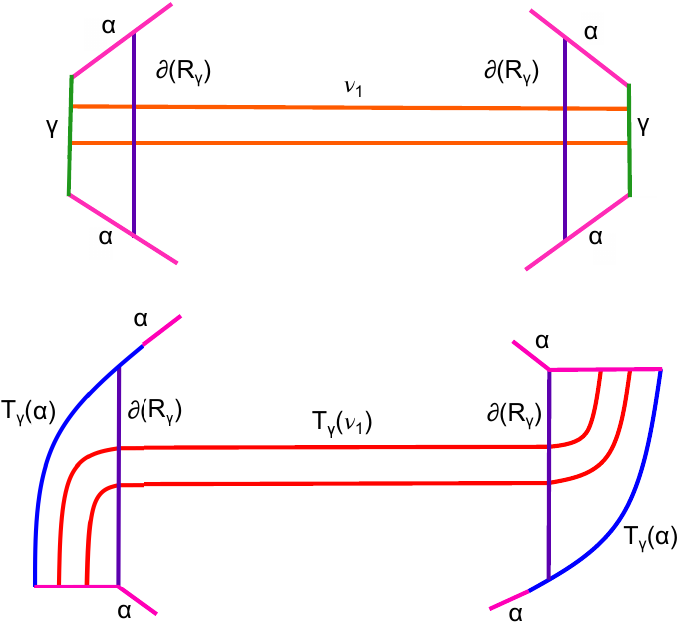}
\caption{Initial efficiency of $T_\gamma(a_1)$ follows from the initial efficiency of $a_1$}
\label{DT18}
\end{figure}

Since $T_\gamma(\alpha), T_\gamma(\nu_1), \nu_2, \nu_1, \alpha$ is already a geodesic we have that $d(T_\gamma(\nu_1), \alpha) =3$. This gives that $T_\gamma(\nu_1)$ is an initially efficient geodesic of distance $4$ from $T_\gamma(\alpha)$ to $\alpha$.

\end{proof}

\vspace*{0.1in}
{\sc Kuwari Mahanta}, Department of Mathematics, Indian Institute of Technology Guwahati, Assam 781039, India, email : kuwari.mahanta@iitg.ac.in

\vspace*{0.1in}
{\sc Sreekrishna Palaparthi} (*corresponding author), Department of Mathematics, Indian Institute of Technology Guwahati, Assam 781039, India, email : passkrishna@iitg.ac.in

\end{document}